\documentclass[12pt]{amsart}
\usepackage[latin1]{inputenc}
\usepackage{amssymb, amsmath, amscd, mathrsfs,marvosym, psfrag,color,a4} 

\input xy
\xyoption{all}
\DeclareMathAlphabet{\mathpzc}{OT1}{pzc}{m}{it}

\newtheorem{theorem}{Theorem}[section]
\newtheorem*{theorem1}{Theorem 1}
\newtheorem*{theorem2}{Theorem 2}

\newtheorem{proposition}[theorem]{Proposition}

\newtheorem{lemma}[theorem]{Lemma}

\theoremstyle{definition}
\newtheorem{definition}[theorem]{Definition}

\theoremstyle{remark}
\newtheorem{remark}[theorem]{Remark}
\newtheorem{remarks}[theorem]{Remarks}

\newcommand{\CC}{{\mathcal C}}
\newcommand{\cD}{{\mathcal D}}

\newcommand{\fh}{{{\mathfrak h}}}

\newcommand{\fg}{{{\mathfrak g}}} 
\newcommand{\fb}{{{\mathfrak b}}}

\newcommand{\fn}{{{\mathfrak n}}}

\newcommand{\fa}{{{\mathfrak a}}}

\newcommand{\fsl}{{{\mathfrak{sl}}}}

\newcommand{\fhd}{\fh^\star}

\newcommand{\DZ}{{\mathbb Z}}

\newcommand{\DN}{{\mathbb N}}
\newcommand{\DQ}{{\mathbb Q}}

\newcommand{\mbe}{{E}}
\newcommand{\mbf}{{F}}

\newcommand{\ch}{{\operatorname{char}\, }}

\newcommand{\Ext}{{\operatorname{Ext}}}
\newcommand{\Hom}{{\operatorname{Hom}}}
\newcommand{\Lie}{{\operatorname{Lie}}}

\newcommand{\supp}{{\operatorname{supp}}\,}

\newcommand{\rad}{{\operatorname{rad}}}
\newcommand{\Dist}{{\operatorname{Dist}}}

\newcommand{\height}{{\operatorname{ht}}}

\newcommand{\ul}{\underline}

\newcommand{\id}{{\operatorname{id}}}

\newcommand{\ad}{{\operatorname{ad}}}

\newcommand{\p}{{\operatorname{p}}}

\newcommand{\comment}[1]{}

\newcommand{\lgl}{\langle}
\newcommand{\rgl}{\rangle}

\newcommand{\LDelta}{{V}}
\newcommand{\WDelta}{{\Delta}}
\begin{document}

\pagenumbering{arabic}
\title[]{Lefschetz operators, Hodge-Riemann forms, and representations} \author[]{Peter Fiebig}
\begin{abstract} 
For a field of characteristic $\ne 2$  we study vector spaces that are graded by the weight lattice of a root system, and are endowed with linear operators in each simple root direction. We show that these data extend to a weight lattice graded semisimple representation of the corresponding Lie algebra, if and only if there exists a bilinear form that satisfies properties (roughly) analogous to those of the Hodge-Riemann forms in complex geometry. In the second part of the article we replace the field by the  $p$-adic integers (with $p\ne 2$) and show that in this case the existence of a certain bilinear form is equivalent to the existence of a  structure of a tilting module for   the associated simply connected $p$-adic Chevalley  group. 
\end{abstract}

\address{Department Mathematik, FAU Erlangen--N\"urnberg, Cauerstra\ss e 11, 91058 Erlangen}
\email{fiebig@math.fau.de}
\maketitle

\section{Introduction}

Let $R$ be a root system, $\Pi\subset R$ a basis and $X$ the weight lattice. Let $K$ be a field of characteristic $\ne 2$, and let $(\fg,\fh)$ be the split semi-simple Lie algebra over $K$ with root system $R$.  Suppose that $M=\bigoplus_{\mu\in X} M_\mu$ is an $X$-graded $K$-vector space. Suppose that for each simple root $\alpha\in\Pi$ we have an operator $F_\alpha\colon M\to M$ that is homogeneous of degree $-\alpha$. One motivation for the present article is to find an answer to the following question: Under what additional conditions is there a $\fg$-module structure on $M$ such  that for each $\alpha$ the homomorphism $F_\alpha$ is the action map of the corresponding  Chevalley generator in $\fg$, and $\fh$ acts on $M_\mu$ via the character associated with  $\mu$? 

In the case that   $R$ is the root system of type $A_1$, $K$ is a field of characteristic $0$, and $M$ is finite dimensional, an answer to the above question is well-known. In that case we can  identify  $X$ with $\DZ$ in such a way that the positive root corresponds to $2$. Let $M=\bigoplus_{n\in\DZ} M_n$ be a finite dimensional, $\DZ$-graded vector space, and $F\colon M\to M$ a homogeneous $K$-linear map of degree $-2$. Recall that $F$ is called a {\em Lefschetz operator} if for any $l\ge 0$ the homorphism $F^l\colon M_l\to M_{-l}$ is an isomorphism. Denote by $e,h,f\in\fsl_2(K)$ the standard generators. Then the following are equivalent.

\begin{enumerate}
\item There is an $\fsl_2(K)$-module structure on $M$ such that $M_n$ is the eigenspace for the action of $h$ with eigenvalue $n$ and the action map of $f$ is $F$.
\item $F$ is a  Lefschetz operator.
\end{enumerate}
Now we give another equivalent condition. For any $m\in\DZ$ let $M_{\{\ge m\}}\subset M$ be the smallest $F$-stable subspace containing $M_n$ for all $n\ge m$. 

\begin{enumerate}\setcounter{enumi}{2}
\item There exists a symmetric non-degenerate bilinear form $(\cdot,\cdot)$ on $M$ with the following properties.
\begin{enumerate}
\item $M_n$ and $M_m$ are orthogonal if $m\ne n$.
\item The restriction of $(\cdot,\cdot)$ to $M_{\{\ge m\}}\times M_{\{\ge m\}}$ is non-degenerate for all $m\in\DZ$. 
\item  If we denote by $E\colon M\to M$ the adjoint of $F$ with respect to $(\cdot,\cdot)$, then $[E,F]|_{M_n}=n\cdot \id_{M_n}$ for all $n\in\DZ$. 
\end{enumerate}
\end{enumerate}
The equivalence of (1), (2) and (3) are the cornerstones of the most common proof of the Hard Lefschetz theorem for complex K\"ahler manifolds. In this case, the bilinear form $(\cdot,\cdot)$ is the {\em Hodge-Riemann form}.  in Section \ref{sec-mot} we provide a proof of the equivalence of (1), (2) and (3) above, as a motivation for the following.

Admittedly, condition (3) is, compared to (1), rather involved. But it has the advantage that it easily generalizes to other root systems.  So let $M=\bigoplus_{\mu\in X} M_\mu$ and $\{F_\alpha\}_{\alpha\in\Pi}$ be as in the first paragraph of this introduction.  We call a subset $I$ of $X$ {\em closed} if it contains with any element $\mu$ all elements that are bigger than $\mu$ with respect to the usual order on $X$. Then we let $M_I$ be the smallest subspace of $M$ that contains $M_\mu$ for all $\mu\in I$ and is stable under all maps $F_\alpha$.  We call a bilinear form $(\cdot,\cdot)$  an {\em HR-form} if the following are satisfied.
\begin{enumerate}
\item $(\cdot,\cdot)$ is non-degenerate and symmetric. 
\item $M_\mu$ and $M_\nu$ are orthogonal if $\mu\ne \nu$.
\item For any closed subset $I$ of $X$,  the restriction of $(\cdot,\cdot)$ to $M_I\times M_I$ is non-degenerate. 
\item  Denote by $E_\alpha\colon M\to M$ the adjoint of $F_\alpha$ with respect to $(\cdot,\cdot)$. Then $[E_\alpha,F_\alpha]|_{M_\mu}=\langle\mu,\alpha^\vee\rangle\cdot \id_{M_\mu}$  for all $\alpha\in\Pi$ and $\mu\in X$ and $[E_\alpha,F_\beta]=0$ if $\alpha\ne\beta$.
\end{enumerate}
The first main result of this article is the following.

\begin{theorem1} Let $K$ be a field of characteristic $\ne 2$. The following are equivalent.
\begin{enumerate}
\item There exists an HR-form on $M$.
\item There exists a structure of a semisimple $X$-graded representation\footnote{See Section \ref{subsec-Xgradrep} for the definition of an $X$-graded representation.} of $\fg$ on $M$  such that  for all $\alpha\in \Pi$, the homomorphism $F_\alpha$ is the action map of a Chevalley generator corresponding to $-\alpha$. 
\end{enumerate}
\end{theorem1}
It might be worth pointing out that in the situation of the theorem above, the required commutation relations between the $E_\alpha$'s and $F_\beta$'s are relations that are certainly satisfied for the standard generators of the Lie algebra,  but that the less obvious Serre relations in the Lie algebra
are a formal consequence of the existence of an HR-form. A similar result in the framework of quantum groups was obtained by Lusztig (cf. \cite[Chapter 1.4]{LQG}).

In the second part of this article we generalize this further. For a prime number $p$ denote by $\DZ_p$ the ring of $p$-adic integers. We now consider $X$-graded $\DZ_p$-modules $M=\bigoplus_{\mu\in X}M_\mu$ that are free of finite rank, together with $\DZ_p$-linear operators $F_\alpha\colon M\to M$ of degree $-\alpha$ for all $\alpha\in\Pi$.  We assume that for any $\alpha\in \Pi$ and  $n\ge0$ we have $F_\alpha^n(M)\subset n! M$. Hence $F_\alpha^{(n)}:=F_\alpha^n/n!$ is a well-defined operator on $M$. For a closed subset $I$ of $X$ we denote by $M_I$ the smallest $\DZ_p$-submodule that is stable under all maps $F_\alpha^{(n)}$ with $n\ge 0$ and $\alpha\in\Pi$ and contains all $M_\mu$ with $\mu\in I$.  We assume that, as a $\DZ_p$-module, $M_I$ is a direct summand in $M$ for all such $I$. 
In this situation  we call a $\DZ_p$-bilinear form  $(\cdot,\cdot)$ on $M$ a  {\em $p$-adic HR-form} if the following are satisfied.

\begin{enumerate}
\item It  is symmetric and non-degenerate.
\item For any closed subset $I$ of $X$, the restriction to $M_I\times M_I$ is faithful. 
\item $M_\mu$ and $M_\nu$ are orthogonal if $\mu\ne \nu$.
\item If we denote by $E_\alpha$ the adjoint of $F_\alpha$ with respect to $(\cdot,\cdot)$, then $[E_\alpha,F_\alpha]|_{M_\mu}=\langle\mu,\alpha^\vee\rangle\cdot \id_{M_\mu}$  for all $\alpha\in\Pi$ and $\mu\in X$ and $[E_\alpha,F_\beta]=0$ if $\alpha\ne\beta$.
\end{enumerate}

Denote by $G_{\DZ_p}$ the split semi-simple, simply connected  algebraic group associated with $R$ over the ring $\DZ_p$. Then the second main result of this article is the following.

\begin{theorem2} Suppose $p\ne 2$. Then the following are equivalent.
\begin{enumerate}
\item There exists a $p$-adic HR-form on $M$.
\item There exists a $G_{\DZ_p}$-module structure on $M$ such that $M=\bigoplus_{\mu\in X}M_\mu$ is the weight decomposition, $F_\alpha$ is the action map of a Chevalley generator in $\Lie(G_{\DZ_p})$ corresponding to $-\alpha$, and with this structure, $M$ is a tilting module.
\end{enumerate}
\end{theorem2}

In a forthcoming article the interested reader will find an algorithm that makes use of the above result. This algorithm takes a dominant weight $\lambda$ and constructs the $X$-graded space underlying the indecomposable tilting module $T_{\DZ_p}(\lambda)$ with highest weight $\lambda$ together with its structure as a module for the algebra of distributions of $G_{\DZ_p}$.  In particular, it yields its character. The algorithm is inductive on the weights, i.e. it constructs $T_{\DZ_p}(\lambda)_\mu$ starting from $\mu=\lambda$ by downwards induction.

\begin{remark}  The main results in this article illustrate the concept that  {\em the world is more semi-simple than it a priori has a right to be\footnote{ G. Williamson, ICM talk 2018}}, and that {\em hidden geometric structure in representation theory is often provided by non-degenerate bilinear forms}\footnote{ G. Williamson, ICM talk 2018}. I would like to thank Tom Braden for sharing his ideas on semi-infinite moment graphs. His construction of bilinear forms on semi-infinite Braden--MacPherson sheaves  was highly inspiring for the present article. 
\end{remark}

\section{Motivation: Lefschetz operators and representations of $\fsl_2$} \label{sec-mot}
As a motivating example we treat in detail the case of the root system $A_1$ and characteristic $0$ field coefficients in this section. So we fix a field $K$ of characteristic $0$, a $\DZ$-graded $K$-vector space $M=\bigoplus_{n\in\DZ} M_n$ of finite dimension, and a $K$-linear homogeneous map
$\mbf\colon M\to M$   of degree  $-2$. Recall that one says that $F$ is a {\em Lefschetz operator} if for any $l\ge 0$ the $l$-th power of $F$ induces an {\em isomorphism} $M_l\xrightarrow{\sim} M_{-l}$ \footnote{In the theory of Lefschetz operators it is much more customary to consider homogeneous operators of degree $+2$. However, since the applications for the theory developed here are mainly of representation theoretic nature, where one is used to the notion of  highest weight  rather than lowest weight modules, we multiply the indices with $-1$.}. For any $m\in\DZ$ we denote by $M_{\{\ge m\}}$ the smallest $F$-stable subspace of $M$ that contains all homogeneous components $M_n$ with $n\ge m$. This is a graded subspace.

Let $e,f,h\in\fsl_{2}(K)$ be the standard generators with relations $[e,f]=h$, $[h,e]=2e$ and $[h,f]=-2f$.

\begin{proposition}\label{prop-motivation} The following are equivalent:
\begin{enumerate}
\item $\mbf$ is a Lefschetz operator.
\item There exists an $\fsl_2(K)$-module structure on $M$ such that for all $n\in\DZ$, $M_n$ is the $h$-eigenspace with eigenvalue $n$, and $F$ is the action map of $f$.
\item  There exists a symmetric $K$-bilinear form $(\cdot,\cdot)$ on $M$ such that the following holds:
\begin{itemize} 
\item The restriction of $(\cdot,\cdot)$ to $M_{\{\ge m\}}\times M_{\{\ge m\}}$ is non-degenerate for all $m\in\DZ$.
\item $M_n$ and $M_m$ are orthogonal for $m\ne n$.
\item  For each $n\in\DZ$ and $v\in M_n$ we have $[\mbe,\mbf](v)=nv$, where $\mbe\colon M\to M$ is the adjoint to $\mbf$ with respect to $(\cdot,\cdot)$.
\end{itemize}
\end{enumerate}
\end{proposition}
\begin{remark}\label{rem-HR} The above statement is well-known. It appears in the most common proof of the Hard Lefschetz theorem for complex K\"ahler manifolds, which is due to Chern (cf. \cite{C}). In this setup,  $M$ is the graded cohomology of such a manifold, shifted in degree by the complex dimension, and $F$ is  given by  the action of the K\"ahler form. The Hard Lefschetz theorem states that $F$ is then a Lefschetz operator in the above sense, and its proof shows that property (3) is satisfied. The bilinear form appearing in (3) is the {\em Hodge-Riemann form}. 
\end{remark}
\begin{proof} Suppose that $\mbf$ is a Lefschetz operator on $M$. For all $m\ge 0$ let $M^{\p}_m\subset M_m$ be the kernel of $\mbf^{m+1}|_{M_m}\colon M_m\to M_{-m-2}$, and set $M^\p:=\bigoplus_{n\ge 0} M_n^\p$. An element in $M_n^\p$ is called   {\em primitive}.  We claim that for each $n\ge 0$ and  $v\in M_n$ there are unique elements $m_{n}\in M_{n}^\p$, $m_{n+2}\in M_{n+2}^\p$,\dots such that
$$
v=m_{n}+\mbf(m_{n+2})+\mbf^2(m_{n+4})+\dots.
$$
First we prove the existence of such a presentation.  Note that if $v\in M_n$ admits a representation as above, then so does $\mbf(v)\in M_{n-2}$.  Using induction on $n$ it is hence enough to show that $M_n=F(M_{n+2})+M_n^\p$.  So let $v\in M_n$ consider $\mbf^{n+1}(v)\in M_{-n-2}$. As $F$ is a Lefschetz operator there is a  $v^\prime\in M_{n+2}$ such that $\mbf^{n+2}(v^\prime)=\mbf^{n+1}(v)$. Then $m_n:=v-\mbf(v^\prime)$ is in the kernel of $\mbf^{n+1}$, hence primitive. 

Now we show that a presentation as above for $v$ is unique. It suffices to show this for $v=0$. So suppose there is a presentation $
0=m_{n}+\mbf(m_{n+2})+\mbf^2(m_{n+4})+\dots.
$ with the property that not all of the $m_{i}$'s are $=0$. Let $r\ge 0$ be maximal with $m_{n+2r}\ne 0$.Then $0=\mbf^{n+r}(0)=\mbf^{n+2r}(m_{n+2r})$, which contradicts the fact that  $\mbf^{n+2r}\colon M_{n+2r}\to M_{-n-2r}$ is an isomorphism. 

Similarly, for $n>0$ and $v\in M_{-n}$ there are unique elements $m_{n}\in M_n^\p$, $m_{n+2}\in M_{n+2}^\p$, \dots such that
$$
v=\mbf^n(m_n)+\mbf^{n+1}(m_{n+2})+\mbf^{n+2}(m_{n+4})+\dots.
$$
This result is obtained from the previous using the isomorphism $\mbf^n\colon M_n\to M_{-n}$. Hence we can write $M=M(1)\oplus \dots\oplus M(l)$, where each $M(s)$ is an $F$-stable, $\DZ$-graded linear subspace of $M$ with the property that there is some homogeneous element  $m_s\in M(s)\cap M_{n_s}^\p$ with $n_s\ge 0$ and such that $m_s,\mbf(m_s),\dots, \mbf^{n_s}(m_s)$ is a basis of $M(s)$. In particular, the dimension of $M(s)$ is $n_s+1$.

Now recall that a finite dimensional $\fsl_2(K)$-module with a diagonalizable action of $h$ splits into a direct sum of copies of irreducible highest weight modules. The highest weight module $L(n)$ (with $n\ge 0$) has a basis $v_1$, $v_2$, \dots, $v_{n+1}$ that satisfies the following: $f.v_j=v_{j+1}$ for $1\le j\le n$ and $f.v_{n+1}=0$, $h.v_j=(n-2j)v_i$ for $1\le j\le n+1$, and $e.v_j=j(n-j+1)v_{j-1}$ for $2\le j\le n+1$ and $e.v_1=0$ (cf. \cite[Section 7.2]{H}, where the basis is differently normalized). 
On each $M(s)$ we can now define the structure of an $\fsl_2(K)$-module as follows. Set $n=n_s+1$ and let $h$ act as multiplication with $i$ on each homogeneous element of degree $i$, $f$ by $\mbf$ and $e$ by setting $e.\mbf^j(m_s):=j(n-j+1)\mbf^{j-1}(m_s)$. This proves that (1) implies (2). 

Now assume (2). Recall that there is an involutive antiautomorphism $\tau$ on $\fsl_2(K)$ with $\tau(e)=f$ and $\tau(h)=h$. For a finite dimensional representation $M$ we denote by $dM$ the contravariant representation. As a vector space, $dM$ is the $K$-linear dual $M^\ast$ of $M$, and the action is given by $(x.f)(m)=f(\tau(x).m)$ for $x\in\fg$, $f\in dM$ and $m\in M$. We then have an $h$-weight decomposition $dM=\bigoplus_{i\in\DZ} (dM)_i$ with $(dM)_i=(M_i)^\ast$. It follows from the above discussion that $dL(n)\cong L(n)$. Any such isomorphism is given by a non-degenerate bilinear form $(\cdot,\cdot)$ on $L(\lambda)$ that has the property that $(x.m,n)=(m,\tau(x).n)$ for all $x\in\fg$ and $m,n\in L(\lambda)$. So  the $h$-eigenspace decomposition is orthogonal and the action map of $e$ is right and left adjoint to the action map of $f$. By the explicit description of a basis above, we deduce that this bilinear form is already determined by its value on a highest weight generator. In particular, it is symmetric. Moreover, for $m\in\DZ$, we have $L(n)_{\{\ge m\}}=L(n)$ if $m\le n$, and $L(n)_{\{\ge m\}}=\{0\}$ if $m>n$. So property (3) follows in the case that $M$ is isomorphic to some $L(n)$. As we noted above, any finite dimensional $h$-diagonalizable $\fsl_2(K)$-module $M$ is isomorphic to a direct sum of copies of various $L(n)$, so we can deduce (3) in the general case.

Property (3) obviously implies (2), and the explicit description of the structure of finite dimensional $\fsl_2(K)$-representations with a diagonalizable $h$-action above shows that (2) implies (1). 
\end{proof}
 The main objective of the present paper is to generalize the equivalence of (2) and (3) in the proposition above in various ways.

\section{Operators in root directions  on $X$-graded $K$-vector spaces}\label{sec-Xgradsp}

Let $R$ be a finite root system in the real vector space $V$, and let $\Pi$ be a basis of $R$. For $\alpha\in \Pi$ denote by $\alpha^\vee\in V^\ast$ its coroot, and let $X\subset V$ be the weight lattice. Let $\le$ be the partial order on $X$ given by $\lambda\le\mu$ if and only if $\mu-\lambda$ can be written as a sum of elements in $\Pi$. We call a subset $I$ of $X$ {\em closed}, if it is upwardly closed with respect to $\le$, i.e. if $\mu\in I$ and $\mu\le\nu$ implies $\nu\in I$. This clearly defines a topology on $X$ with the open sets being the downwardly closed sets.

\subsection{Graded  spaces with operators}
We fix a field $K$ of arbitrary characteristic. We will later assume that $\ch K\ne 2$.  Let $M=\bigoplus_{\mu\in X} M_\mu$ be an $X$-graded $K$-vector space. The {\em support of $M$} is $\supp M=\{\mu\in X\mid M_\mu\ne\{0\}\}$. The homogeneous elements in $M$ are called {\em weight vectors} and  $\mu$   is called the {\em weight} of $w$ if $w\in M_\mu$.  For any $\alpha\in\Pi$ we fix  a homogeneous $K$-linear map $F_\alpha\colon M\to M$  of degree $-\alpha$. 

We say that  a  subset $S$ of $M$ is {\em $F$-stable} if $F_\alpha(S)\subset S\cup\{0\}$ for all $\alpha\in\Pi$. For a  subset $T$ of $M$ we denote by $\lgl T\rgl_F$ the smallest $F$-stable linear subspace of $M$ that contains $T$.   
Note that if $T\subset M$ is a graded subset, i.e. if it contains with any element all of its homogeneous summands, then $\lgl T\rgl_F$ is a graded subspace.  
For a closed subset $I$ of $X$ we set $M_I:=\lgl \bigoplus_{\mu\in I} M_\mu\rgl_F\subset M$. We will always assume the following. 

\begin{enumerate}
\item[$(\ast)_K$] There is a finite subset $T$ of $M$  such that $M=\lgl T\rgl_F$. 
\end{enumerate}
The above implies that each graded component $M_\mu$ is finite dimensional, and that the support of $M$ in $X$ is bounded from above, i.e. there exist $\gamma_1$,\dots,$\gamma_l\in X$ such that $M_\mu\ne 0$ implies that $\mu\le \gamma_i$ for some $i\in\{1,\dots,l\}$. 

\begin{definition}  We say that $M$ is {\em $F$-cyclic with highest weight $\lambda\in X$} if there is an element $m\in M_\lambda$, $m\ne 0$, and $M=\lgl \{m\}\rgl_F$. 
\end{definition}
If $M$ is $F$-cyclic, then $M_\lambda$ is one-dimensional and $M_\mu\ne\{0\}$ implies $\mu\le\lambda$.

\subsection{HR-forms}

Let $(\cdot,\cdot)\colon M\times M\to K$ be a $K$-bilinear form on $M$.

\begin{definition}\label{def-formK} 
 We say that  $(\cdot,\cdot)$ is an {\em HR-form} on $M$ if the following holds:
\begin{enumerate}
\item $(\cdot,\cdot)$ is symmetric. 
\item For any closed subset $I$, the restriction of $(\cdot,\cdot)$ to $M_I\times M_I$ is non-degenerate.
\item $M_\mu$ and $M_\nu$ are orthogonal for $\mu\ne\nu$. 
\item For $\alpha,\beta\in\Pi$ we have
\begin{itemize}
\item  $[\mbe_\alpha,\mbf_\alpha](v)=\langle\mu,\alpha^\vee\rangle v$ for all $\mu\in X$ and  $v\in M_\mu$,
\item $[\mbe_\alpha,\mbf_\beta]=0$ if $\alpha\ne\beta$,
\end{itemize}
where $\mbe_\alpha\colon M\to M$ is the adjoint of $\mbf_\alpha$ with respect to $(\cdot,\cdot)$\footnote{Note that (2) for $I=X$ implies that $(\cdot,\cdot)\colon M\times M\to K$ is non-degenerate, and (3) implies that  $E_\alpha$ is homogeneous of degree $+\alpha$, hence $[E_\alpha,F_\alpha]$ preserves the grading.}.  
\end{enumerate}
\end{definition}

\begin{remark} Suppose that $K$ is a field of characteristic $0$, that $M$ is finite dimensional and that there exists an HR-form on $M$. Fix a simple root $\alpha$. Then $M$ splits into the direct sum of the $F_{\alpha}$-stable subspaces $M_{S}:=\bigoplus_{\gamma\in S} M_{\gamma}$, where $S$ runs through the set of {\em $\alpha$-strings} $X/\DZ \alpha$ in $X$. We consider each $M_{S}$ as $\DZ$-graded with $(M_S)_n:=M_\gamma$ if there is a (unique) $\gamma\in S$ with $\lgl\gamma,\alpha^\vee\rgl=n$, and $=0$ otherwise. Then Proposition \ref{prop-motivation} implies that $F_\alpha$ is a Lefschetz operator on each $M_S=\bigoplus_{n\in\DZ} (M_S)_n$. 
\end{remark}

\begin{remark}\label{rem-decomp1}  Let $(\cdot,\cdot)$ be an HR-form on $M$, and suppose that $M=M_1\oplus M_2$ is a decomposition into $F$-stable $X$-graded subspaces and that $M_1$ and $M_2$ are orthogonal with respect to $(\cdot,\cdot)$. Then the restrictions of $(\cdot,\cdot)$ to $M_1$ and $M_2$ are HR-forms again.
\end{remark}

\begin{lemma} \label{lemma-HRform} Suppose that $M$ is $F$-cyclic with highest weight $\lambda$ and that $(\cdot,\cdot)$ is a bilinear form that is non-degenerate on $M$ and satisfies properties (1), (3) and (4) of Definition \ref{def-formK}. Then property (2) is also satisfied. 
\end{lemma}
\begin{proof} Let $I\subset X$ be a closed subset. As $M$ is cyclic with highest weight $\lambda$, we have $M_\mu=\{0\}$ unless $\mu\le\lambda$. Hence  $M_I=M$ if $\lambda\in I$, and $M_I=\{0\}$ in all other cases.
\end{proof}

\subsection{Simple root paths}
Our goal now is to construct a {\em universal} $F$-cyclic $X$-graded space with operators that admits an HR-form and  has a highest weight $\lambda\in X$. 

\begin{definition} A {\em simple root path} is a sequence $\ul\alpha:=(\alpha_1,\dots,\alpha_l)$ of not necessarily distinct simple roots $\alpha_1,\dots,\alpha_l\in\Pi$. Its  {\em length} is $l(\ul\alpha)=l$ and its {\em height}  is $\height(\ul\alpha):=\alpha_1+\dots+\alpha_l\in\DZ_{\ge0} \Pi\subset X$. We denote the {\em reversed simple root path} by  $\ul\alpha^r:=(\alpha_l,\dots,\alpha_1)$. For $i\in\{1,\dots,l\}$ we set $\ul\alpha^{(i)}:=(\alpha_1,\dots,\alpha_{i-1},\alpha_{i+1},\dots,\alpha_l)$.
\end{definition}
Since the steps in a path are given by simple roots, paths of the same height have the same length and they differ only by the order in which the simple roots appear. Clearly  $\height(\ul\alpha)=\height(\ul\alpha^r)$ and $\height(\ul\alpha^{(i)})=\height(\ul\alpha)-\alpha_i$.

Fix $\lambda\in X$ and let $P(\lambda)_\mu$ be the $K$-vector space that has as a basis  all paths of height $\lambda-\mu$. Set $P(\lambda):=\bigoplus_{\mu\in X} P(\lambda)_\mu$.  Hence we obtain $P(\lambda)$ from $P(0)$ by a simple grading shift by $-\lambda$, and $P(0)_\mu$ has a basis consisting of all simple root paths of weight $-\mu$. For $\alpha\in\Pi$  we denote by  $\varphi_\alpha\colon P(\lambda)\to P(\lambda)$ the $K$-linear map that maps a path $(\gamma_1,\dots,\gamma_l)$ to the path $(\alpha,\gamma_1,\dots,\gamma_l)$. Then $\varphi_\alpha$ is homogeneous of degree $-\alpha$. 
We denote by $\epsilon_\alpha\colon P(\lambda)\to P(\lambda)$  the $K$-linear map defined on paths by
$$
\epsilon_{\alpha}(\gamma_1,\dots,\gamma_l)=\sum_{i, \gamma_i=\alpha}\lgl\lambda-\gamma_{i+1}-\dots-\gamma_l,\alpha^\vee\rgl\ul\gamma^{(i)}.
$$
So $\epsilon_\alpha$ sends a path of length $l$ to a linear combination of paths of length $l-1$, and it maps $P(\lambda)_\mu$ to $P(\lambda)_{\mu+\alpha}$, i.e. it is homogeneous of degree $+\alpha$.

\begin{lemma}\label{lemma-comrel} Let $\alpha\in\Pi$. Then $[\epsilon_\alpha,\varphi_\alpha]|_{P(\lambda)_\mu}=\langle\mu,\alpha^\vee\rangle\cdot\id_{P(\lambda)_\mu}$ for all $\mu\in X$. If  $\beta\in\Pi$, $\alpha\ne\beta$, then $[\epsilon_\alpha,\varphi_\beta]=0$.
\end{lemma}
\begin{proof} Let $\ul\gamma=(\gamma_1,\dots,\gamma_l)$ be a path of height $\lambda-\mu$. Then
\begin{align*}
\epsilon_\alpha\varphi_\alpha(\ul\gamma)&=\epsilon_\alpha(\alpha,\gamma_1,\dots,\gamma_l) \\
&=\lgl\lambda-\sum_{j=1}^l\gamma_j ,\alpha^\vee\rgl (\gamma_1,\dots,\gamma_l) +\varphi_{\alpha}\epsilon_{\alpha}(\ul\gamma) \\
&=\lgl\mu,\alpha^\vee\rgl\ul\gamma+\varphi_\alpha\epsilon_\alpha(\ul\gamma)
\end{align*}
as $\lambda-\sum_{j=1}^l\gamma_j=\mu$. If $\alpha\ne\beta$, then the first summand in the sums above must be deleted, and hence $
\epsilon_\alpha\varphi_\beta(\ul\gamma)=\varphi_\beta\epsilon_\alpha(\ul\gamma)$.
\end{proof}

 For a path $\ul\alpha=(\alpha_1,\dots,\alpha_l)$ set $\varphi_{\ul\alpha}:=\varphi_{\alpha_1}\circ\dots\circ\varphi_{\alpha_l}$ and $\epsilon_{\ul\alpha}:=\epsilon_{\alpha_1}\circ\dots\circ\epsilon_{\alpha_l}$. 
 
\begin{lemma} \label{lemma-comrel2} For a path ${\ul\alpha}=(\alpha_1,\dots,\alpha_l)$, $\beta\in\Pi$  and $v\in P(\lambda)_\mu$ we have
\begin{align*}
[\epsilon_{\ul\alpha},\varphi_\beta](v)&=\sum_{i,\alpha_i=\beta}\lgl\mu+\alpha_{i+1}+\dots+\alpha_{l},\beta^\vee\rgl \epsilon_{\ul\alpha^{(i)}}(v).
\end{align*}

\end{lemma}
\begin{proof} This is an easy exercise using the commutation relations in Lemma \ref{lemma-comrel}.
\end{proof}

\subsection{A bilinear form on $P(\lambda)$}
Define a bilinear form $(\cdot,\cdot)_{P(\lambda)}$ on $P(\lambda)$ by bilinear extension of the formulas 
$$
(\ul\alpha,\ul\beta)_{P(\lambda)}=\begin{cases}
0,& \text{ if $\height(\ul\alpha)\ne\height(\ul\beta)$},\\
\epsilon_{{\ul\alpha^r}}(\ul\beta), &\text{ if $\height(\ul\alpha)=\height(\ul\beta)$},
\end{cases}
$$
for  paths $\ul\alpha$ and $\ul\beta$. Note that $\height(\ul\alpha)=\height(\ul\beta)$ implies that $\epsilon_{{\ul\alpha^r}}(\ul\beta)$ is a $K$-multiple of the empty path, hence it can be identified with a scalar. 

\begin{lemma}\label{lemma-pathsbiform} 
\begin{enumerate}
\item The bilinear form $(\cdot,\cdot)_{P(\lambda)}$ is symmetric.
\item For  all $\alpha\in\Pi$ the map $\epsilon_\alpha$ is adjoint to $\varphi_\alpha$ with respect to $(\cdot,\cdot)_{P(\lambda)}$.
\end{enumerate}
\end{lemma}
\begin{proof} (1) Let $\ul\alpha$ and $\ul\beta$ be paths. If $\height(\ul\alpha)\ne\height(\ul\beta)$, then $(\ul\alpha,\ul\beta)=0=(\ul\beta,\ul\alpha)$. So suppose that $\ul\alpha$ and $\ul\beta$ are of equal height and hence of equal length. We then have to show that  $\epsilon_{\ul\alpha^r}(\ul\beta)=\epsilon_{\ul\beta^r}(\ul\alpha)$ or 
$$
\epsilon_{\ul\alpha^r}\varphi_{\ul\beta}(\emptyset)=\epsilon_{{\ul\beta}^r}\varphi_{\ul\alpha}(\emptyset),
$$
where by $\emptyset\in P(\lambda)_\lambda$ we denote the empty path. Suppose $\ul\alpha=(\alpha_1,\dots,\alpha_l)$ and $\ul\beta=(\beta_1,\dots,\beta_l)$.  Using Lemma \ref{lemma-comrel2}  we obtain
\begin{align*}
\epsilon_{\ul\alpha^r}\varphi_{\ul\beta}(\emptyset)&=\epsilon_{\ul\alpha^r}\varphi_{\beta_1}\varphi_{\ul\beta^{(1)}}(\emptyset)\\
&=\sum_{i,\alpha_i=\beta_1}\lgl\mu+\alpha_{i-1}+\dots+\alpha_{1},\beta_1^\vee\rgl\epsilon_{{\ul\alpha}^{(i)r}}\varphi_{\ul\beta^{(1)}}(\emptyset) \\
&\quad + \varphi_{\beta_1}\epsilon_{\ul\alpha^r}\varphi_{\ul\beta^{(1)}}(\emptyset),
\end{align*}
where $\mu:=\lambda-\beta_2-\dots-\beta_l$ is the weight of $\varphi_{\ul\beta^{(1)}}(\emptyset)$. Now 
$\epsilon_{\ul\alpha^r}\varphi_{\ul\beta^{(1)}}(\emptyset)$ is a vector of weight $\lambda+\beta_1$, hence zero. Note that for all $i$ with $\beta_1=\alpha_i$ we have $\height(\ul\alpha^{(i)r})=\height(\ul\beta^{(1)})$. Since $l(\ul\beta^{(1)})=l(\ul\alpha^{(i)r})=l-1$ we can use an inductive argument and write $\epsilon_{{\ul\alpha}^{(i)r}}\varphi_{\ul\beta^{(1)}}(\emptyset)=\epsilon_{\ul\beta^{(1)r}}\varphi_{\ul\alpha^{(i)}}(\emptyset)$. So we obtain 
 $$
\epsilon_{\ul\alpha^r}\varphi_{\ul\beta}(\emptyset)=\sum_{i,\alpha_i=\beta_1}\lgl\mu+\alpha_{i-1}+\dots+\alpha_{1},\beta_1^\vee\rgl\epsilon_{{\ul\beta}^{(1)r}}\varphi_{\ul\alpha^{(i)}}(\emptyset).
$$

On the other hand, by the definition of the $\epsilon$-maps, 
\begin{align*}
\epsilon_{\ul\beta^r}\varphi_{\ul\alpha}(\emptyset)&=\epsilon_{\ul\beta^{(1)r}}\epsilon_{\beta_1}\varphi_{\ul\alpha}(\emptyset)\\
&=\epsilon_{\ul\beta^{(1)r}}(\sum_{i,\alpha_i=\beta_1}\lgl\lambda-\alpha_{i+1}-\dots-\alpha_l,\beta_1^\vee\rgl \varphi_{\ul\alpha^{(i)}})(\emptyset)\\
&\quad + \epsilon_{\ul\beta^{(1)r}}\varphi_{\ul\alpha}\epsilon_{\beta_1}(\emptyset) \\
&=\sum_{i,\alpha_i=\beta_1}\lgl\lambda-\alpha_{i+1}-\dots-\alpha_l,\beta_1^\vee\rgl\epsilon_{\ul\beta^{(1)r}} \varphi_{\ul\alpha^{(i)}}(\emptyset),
\end{align*}
as $\epsilon_{\beta_1}(\emptyset)=0$ by definition. 
Now for all $i$ with $\alpha_i=\beta_1$ we have $\beta_2+\dots+\beta_l=\alpha_1+\dots+\alpha_{i-1}+\alpha_{i+1}+\dots+\alpha_l$, hence $\lambda-\alpha_{i+1}-\dots-\alpha_l=\mu+\alpha_{i-1}+\dots+\alpha_1$. A comparison of the above equations yields $\epsilon_{\ul\alpha^r}\varphi_{\ul\beta}(\emptyset)=\epsilon_{\ul\beta^r}\varphi_{\ul\alpha}(\emptyset)$, which is what we wanted to show.

(2)  We need to check $(\epsilon_\alpha(v),w)=(v,\varphi_\alpha(w))$ for all  $v,w\in P(\lambda)$. It is sufficient to check this in the case that $v=\ul\beta=(\beta_1,\dots,\beta_l)$ and $w=\ul\gamma=(\gamma_1,\dots,\gamma_r)$.   We can also assume that the $\height(\ul\beta)-\alpha=\height(\ul\gamma)$,  because otherwise both sides of the equation vanish. Now
\begin{align*}
(\epsilon_{\alpha}(\ul\beta),\ul\gamma)&=(\sum_{i,\beta_i=\alpha}\lgl\lambda-\beta_{i+1}-\dots-\beta_l,\alpha^\vee\rgl \ul\beta^{(i)},\ul\gamma)\\
&=\sum_{i,\beta_i=\alpha}\lgl\lambda-\beta_{i+1}-\dots-\beta_l,\alpha^\vee\rgl(\ul\beta^{(i)},\ul\gamma)\\
&=\sum_{i,\beta_i=\alpha}\lgl\lambda-\beta_{i+1}-\dots-\beta_l,\alpha^\vee\rgl\epsilon_{\ul\beta^{(i)r}}\varphi_{\ul\gamma}(\emptyset).
\end{align*}
On the other hand,
\begin{align*}
(\ul\beta,\varphi_{\alpha}(\ul\gamma))&=\epsilon_{\ul\beta^r}\varphi_{\alpha}\varphi_{\ul\gamma}(\emptyset)\\
&=\sum_{i,\beta_i=\alpha}\lgl\mu+\beta_1+\dots+\beta_{i-1},\alpha^\vee\rgl\epsilon_{\ul\beta^{(i)r}}\varphi_{\ul\gamma}(\emptyset)\\
&\quad +\varphi_{\alpha}\epsilon_{\ul\beta^r}\varphi_{\ul\gamma}(\emptyset) \\
&=\sum_{i,\beta_i=\alpha}\lgl\mu+\beta_1+\dots+\beta_{i-1},\alpha^\vee\rgl\epsilon_{\ul\beta^{(i)r}}\varphi_{\ul\gamma}(\emptyset)
\end{align*}
(as $\epsilon_{\ul\beta^r}\varphi_{\ul\gamma}(\emptyset)=0$ by weight considerations), where $\mu=\lambda-\height(\gamma)=\lambda-\height(\beta)+\alpha$ is the weight of $\varphi_{\ul\gamma}(\emptyset)$. For any $i$ with $\beta_i=\alpha$ we deduce
\begin{align*}
\lambda-\beta_{i+1}-\dots-\beta_l&=\mu+\beta_1+\dots+\beta_i-\alpha\\
&=\mu+\beta_1+\dots+\beta_{i-1}.
\end{align*}
A comparison of the obtained formulas yields $(\epsilon_{\alpha}(\ul\beta),\ul\gamma)=(\ul\beta,\varphi_{\alpha}(\ul\gamma))$.
\end{proof}

\subsection{The standard objects $\LDelta(\lambda)$} 
We   define $\LDelta(\lambda):=P(\lambda)/\rad (\cdot,\cdot)_{P(\lambda)}$ and denote by $(\cdot,\cdot)_{\LDelta(\lambda)}$ the bilinear form  on $\LDelta(\lambda)$ induced by $(\cdot,\cdot)_{P(\lambda)}$. This is a non-degenerate form on $\LDelta(\lambda)$. As the weight space decomposition of $P(\lambda)$ is orthogonal with respect to $(\cdot,\cdot)_{P(\lambda)}$, the radical is a graded subspace. Hence $\LDelta(\lambda)$ inherits a grading $\LDelta(\lambda)=\bigoplus_{\mu\in X}\LDelta(\lambda)_\mu$.  Let $\alpha\in\Pi$. As $\varphi_{\alpha}$ has an adjoint homomorphism $\epsilon_{\alpha}$, it stabilizes the radical, hence induces  a $K$-linear operator $F_\alpha$ on $\LDelta(\lambda)$. Clearly $F_\alpha$ is homogeneous of degree $-\alpha$.  

\begin{lemma}  Let  $\lambda\in X$.
\begin{enumerate}
\item  The data $\LDelta(\lambda)$ and $\{F_\alpha\}_{\alpha\in\Pi}$ satisfy  $(\ast)_K$.
\item $\LDelta(\lambda)$ is $F$-cyclic with highest weight $\lambda$.
\item  $(\cdot,\cdot)_{\LDelta(\lambda)}$ is an HR-form on $\LDelta(\lambda)$.
\end{enumerate}
\end{lemma} 

\begin{proof}  As $\LDelta(\lambda)$ is generated over the $F$-maps by the coset of the empty path, which is of weight $\lambda$,   $(\ast)_K$  is satisfied and $\LDelta(\lambda)$ is $F$-cyclic with highest weight $\lambda$.  The bilinear form $(\cdot,\cdot)_{\LDelta(\lambda)}$ is non-degenerate and it is symmetric because $(\cdot,\cdot)_{P(\lambda)}$ is symmetric by Lemma \ref{lemma-pathsbiform}. The weight space decomposition is orthogonal because that property holds for $P(\lambda)$ by construction.  From  Lemma \ref{lemma-pathsbiform} we deduce that the linear endomorphism $E_\alpha$ on $\LDelta(\lambda)$ that is adjoint to  $F_\alpha$ is induced from $\epsilon_\alpha$. Then Lemma \ref{lemma-comrel} implies that  $E_\alpha$ and $F_\alpha$ satisfy the  commutation relations in Definition \ref{def-formK}. Hence properties (1), (3) and (4) of an HR-form are satisfied. As $\LDelta(\lambda)$ is $F$-cyclic with highest weight $\lambda$, we can use Lemma \ref{lemma-HRform} to deduce that  $(\cdot,\cdot)_{\LDelta(\lambda)}$ is an HR-form. 
\end{proof}

\subsection{The universal property of $\LDelta(\lambda)$}
Suppose that  $M$, $\{F_\alpha\}_{\alpha\in\Pi}$ satisfy  $(\ast)_K$  and fix an HR-form  $(\cdot,\cdot)$  on $M$. As before we denote by $E_\alpha\colon M\to M$ the adjoint of $F_\alpha$ with respect to $(\cdot,\cdot)$. 
For any simple root path $\ul\alpha=(\alpha_1,\dots,\alpha_l)$ we define 
$$
F_{\ul\alpha}:=F_{\alpha_1}\cdots F_{\alpha_l}\colon M\to M,\quad E_{\ul\alpha}:=E_{\alpha_1}\cdots E_{\alpha_l}\colon M\to M.
$$

\begin{lemma}\label{lemma-comrelcyc} Let $\nu\in X$ and   $m\in M_\nu$. For any simple root path $\ul\beta=(\beta_1,\dots,\beta_l)$ and $\alpha\in\Pi$ we have
$$
[E_\alpha, F_{\ul\beta}](m)=\sum_{i,\alpha=\beta_i}\lgl\nu-\beta_{i+1}-\dots-\beta_l,\alpha^\vee\rgl F_{\ul\beta^{(i)}}(m).
$$
\end{lemma}
\begin{proof} This is an easy consequence of the commutation relations in the definition of an HR-form. 
\end{proof}


\begin{proposition} \label{prop-univprop} Let $\lambda\in X$ and  suppose that $M$ is $F$-cyclic with highest weight $\lambda$. Then there exists an isomorphism   $f\colon \LDelta(\lambda)\to M$ of $X$-graded $K$-vector spaces  such that $f\circ F_\alpha=F_\alpha\circ f$ for all $\alpha\in\Pi$, and a non-zero $c\in K$ with  $(f(x),f(y))=c(x,y)_{\LDelta(\lambda)}$ for all $x,y\in\LDelta(\lambda)$.
\end{proposition}

\begin{proof} Let $m\in M_\lambda$ be a non-zero element. Then  $c:=(m,m)\ne 0$ and  $M=\lgl\{m\}\rgl_F$. We first construct a homomorphism $\tilde f\colon P(\lambda)\to M$ by  linear extension of the rules $\tilde f(\ul\alpha)=F_{\ul\alpha}(m)$ for all paths $\ul\alpha$. Clearly, $\tilde f$ is a surjective $X$-graded homomorphism and $\tilde f\circ\varphi_\alpha=F_\alpha\circ\tilde f$ for all $\alpha\in\Pi$. The definition of $\epsilon_{\alpha}$ and Lemma \ref{lemma-comrelcyc}  (note that $E_\alpha(m)=0$) show that also $\tilde f\circ\epsilon_{\alpha}=E_\alpha\circ\tilde f$. Now $(\tilde f(\ul\alpha),\tilde f(\ul\beta))=0=c(\ul\alpha,\ul\beta)_{P(\lambda)}$ if $\height(\ul\alpha)\ne\height(\ul\beta)$. If $\height(\ul\alpha)=\height(\ul\beta)$, then 
\begin{align*}
(\tilde f(\ul\alpha),\tilde f(\ul\beta))&=(F_{\ul\alpha}(m),F_{\ul\beta}(m))\\
&=(m,E_{\ul\alpha^r}F_{\ul\beta}(m))
\end{align*}
and $E_{\ul\alpha^r}F_{\ul\beta}(m)=\xi m$ for some $\xi\in K$. Likewise, $(\ul\alpha,\ul\beta)_{P(\lambda)}=\zeta$, where $\zeta\in K$ is such that $\epsilon_{\ul\alpha^r}\varphi_{\ul\beta}(\emptyset)=\zeta\emptyset$. As both $m$ and $\emptyset$ are annihilated by the $E_{\alpha}$'s and the $\epsilon_{\alpha}$'s, and as the commutation relations between the $E_{\alpha}$'s and  the $F_{\beta}$'s are the same as between the $\epsilon_{\alpha}$'s and the  $\varphi_{\beta}$'s, we can deduce $\xi=\zeta$, hence
$
(\tilde f(\ul\alpha),\tilde f(\ul\beta))=(\ul\alpha,\ul\beta)_{P(\lambda)}
$
also in the case of the same height. By bilinear extension we obtain
$$
(\tilde f(x),\tilde f(y))=c(x,y)_{P(\lambda)}
$$ 
for all $x,y\in P(\lambda)$. 

As $\tilde f$ is surjective and $(\cdot,\cdot)$ is non-degenerate, the radical of $(\cdot,\cdot)_{P(\lambda)}$ is contained in the kernel of $\tilde f$. So $\tilde f$ induces a homomorphism $f\colon \LDelta(\lambda)\to M$. It has the property $(f(x),f(y))=c(x,y)_{\LDelta(\lambda)}$ for all $x,y\in \LDelta(\lambda)$. As $(\cdot,\cdot)_{\LDelta(\lambda)}$ is non-degenerate, $f$ is injective. It is surjective as $\tilde f$ was surjective, so it is an isomorphism of $K$-vector spaces. By construction, it is graded and commutes with the $F$-maps. 
\end{proof}

\subsection{The decomposition into $F$-cyclic subspaces}
Let $M$, $\{F_\alpha\}_{\alpha\in\Pi}$, $(\cdot,\cdot)$, $\{E_\alpha\}_{\alpha\in\Pi}$ be as in the preceding subsection.  We say that a subset $T$ of $M$ is {\em $E$-stable} if $E_\alpha(T)\subset T\cup\{0\}$ for all $\alpha\in\Pi$.

\begin{lemma} \label{lemma-Estab} Suppose that $T$ is $E$-stable. Then $\lgl T\rgl_F$ is $E$-stable. 
\end{lemma}
\begin{proof} Note that $\lgl T\rgl_F$ is generated, as a vector space, by the sets $F_{\ul\beta}(T)$, where $\ul\beta$ runs through all simple root paths. The statement follows now from Lemma \ref{lemma-comrelcyc}. 
\end{proof}

Let $I$ be a closed subset of $X$. As $(\cdot,\cdot)$ is non-degenerate on $M\times M$ and on $M_I\times M_I$, and as, the weight space decomposition is an orthogonal decomposition into finite dimensional spaces,  we obtain $M=M_I\oplus M_I^\perp$, where $M_I^\perp=\{m\in M\mid (m,n)=0 \text{ for all $n\in M_I$}\}$.

\begin{lemma} \label{lemma-orthdec} The following holds.
\begin{enumerate}
\item Both spaces $M_I$ and $M_I^\perp$ are $E$- and $F$-stable subspaces of $M$.
\item The restriction of $(\cdot,\cdot)$ to either $M_I$ or $M_I^\perp$ is again an HR-form.
\end{enumerate}
\end{lemma}
\begin{proof}  By definition $M_I=\lgl\bigoplus_{\mu\in I} M_\mu\rgl_F$ is $F$-stable. As $I$ is a closed subset, $\bigoplus_{\mu\in I} M_\mu$ is an $E$-stable subset of $M$. Hence Lemma \ref{lemma-Estab} implies that $M_I$ is $E$-stable as well.  As the $E_\alpha$'s are the adjoint of the $F_\alpha$'s, we deduce that $M_I^\perp$ is also $E$- and $F$-stable. Hence we proved (1). Claim (2) follows now from Remark \ref{rem-decomp1}. 
\end{proof} 

\begin{proposition} \label{prop-isotyp} Suppose that $\ch K\ne 2$. Let $(\cdot,\cdot)$ be an HR-form on $M$. 
Then there exists some $n>0$ and a family $\{M_i\}_{i=1,\dots,n}$ of graded  subspaces of $M$ that satisfies the following.
\begin{enumerate}
\item Each $M_i$ is $F$-stable and $F$-cyclic. 
\item $M=\bigoplus_{i=1}^nM_i$ as a $K$-vector space.
\item $M_i$ and $M_j$ are orthogonal with respect to $(\cdot,\cdot)$ if $i\ne j$.
\item  The restriction of $(\cdot,\cdot)$ to each $M_i$ is an HR-form again.
\end{enumerate}

\end{proposition}

\begin{proof} Suppose that $\lambda\in X$ is a maximal weight of $M$. Set $I:=\{\mu\in X\mid\lambda\le\mu\}$. Then $I$ is closed. From Lemma \ref{lemma-orthdec} we obtain an orthogonal decomposition $M=M_I\oplus M_I^{\perp}$ into $E$- and $F$-stable subspaces such that the restriction of $(\cdot,\cdot)$ to either $M_I$ or $M_I^\perp$ is an HR-form again. Now $M_I$ is generated by $(M_I)_\lambda$, and each  weight $\mu$ of $M_I^{\perp}$ satisfies $\lambda\not\le\mu$. Using induction and property $(\ast)_K$ we see that it is enough to prove the statement in the case that $M$ is generated by $M_\lambda$, where $\lambda$ is a maximal weight of $M$.

As the bilinear form $(\cdot,\cdot)$ is non-degenerate on $M_\lambda$ and the characteristic of $K$ is $\ne 2$, there exists an orthogonal basis $m_1$, \dots, $m_n$ of $M_\lambda$ such that $(m_i,m_i)\ne 0$. Set $M_i:=\lgl \{m_i\}\rgl_F$. So $M_i$ is an $F$-cyclic subspace. Clearly, $M$ is generated by the $M_i$. The maximality of $\lambda$ implies that each $m_i$ is annihilated by all $E_\alpha$, hence Lemma \ref{lemma-Estab} implies that each $M_i$ is $E$-stable. We claim that $M_i$ and $M_j$ are orthogonal  for $i\ne j$. First, for any simple root path $\ul\alpha$ we have $(m_i,F_{\ul\alpha}(m_j))=(E_{\ul\alpha^r}(m_i),m_j)$. Now this is $=0$ if $\ul\alpha$ is not the empty path, 
 and for $\ul\alpha=\emptyset$ it is $=0$ as $m_i$ and $m_j$ are orthogonal. Hence $(m_i,M_j)=\{0\}$. For any simple root path $\ul\alpha$ we deduce $(F_{\ul\alpha}(m_i),M_j)=(m_i,E_{\ul\alpha}(M_j))=\{0\}$ as $M_j$ is $E$-stable. Hence $(M_i,M_j)=\{0\}$. Hence the $M_i$ are mutually orthogonal. It follows that their sum is direct, i.e. $M=\bigoplus_{i=1}^l M_i$. By Remark \ref{rem-decomp1} the restriction of $(\cdot,\cdot)$ to $M_i\times M_i$ is again an HR-form. 
\end{proof}

Proposition \ref{prop-isotyp} together with Proposition \ref{prop-univprop} implies the following.
\begin{proposition}\label{prop2-isotyp} Suppose that $\ch K\ne 2$. Suppose $M$ and $\{F_\alpha\}_{\alpha\in\Pi}$ satisfy  $(\ast)_K$  and suppose that these data admit an HR-form. Then there is some $n>0$ and   $\lambda_i\in X$ for $i\in \{1,\dots,n\}$ and an $X$-graded $K$-linear isomorphism $M\xrightarrow{\sim} \bigoplus_{i=1}^n\LDelta(\lambda_i)$ that commutes with the $F$-homomorphisms on both sides.
\end{proposition}
Our next step is to show that each $V(\lambda)$ comes from an irreducible highest weight module for the Lie algebra associated with $R$ via forgetting structure. 
\section{The connection to representation theory}\label{sec-repthK}
We fix the data $R$, $\Pi$, and $K$ from Section \ref{sec-Xgradsp}.
Let $(\fg,\fh)$ be the split semisimple Lie algebra over $K$ with root system $R$. Recall that $\fg$ can be constructed as the $K$-Lie algebra with generators $\{e_\alpha, f_\alpha, h_\alpha \mid \alpha\in \Pi\}$ and the following relations. 
\begin{align*}
[h_\alpha,h_\beta]&=0, &&&&\\
[h_\alpha, e_\beta]&=\lgl\beta,\alpha^\vee\rangle e_\beta, & [h_\alpha, f_\beta]&=-\lgl\beta,\alpha^\vee\rangle f_\beta, \\
[e_\alpha,f_\alpha]&=h_\alpha,& [e_\alpha,f_\beta]&=0   \text{ if $\alpha\ne\beta$},&&\\
(\ad\, e_\alpha)^{-\lgl\alpha,\beta^\vee\rgl+1}(e_\beta)&=0,&  (\ad\, f_\alpha)^{-\lgl\alpha,\beta^\vee\rgl+1}(f_\beta)&=0  \text{ if $\alpha\ne\beta$}.&&
\end{align*}
The  $h_\alpha$ are a basis of the Cartan subalgebra $\fh$. We denote by $\fn^-$ the subalgebra generated by all $f_\alpha$, $\alpha\in\Pi$, and by $\fn^+$ the subalgebra generated by all $e_\alpha$, $\alpha\in\Pi$.  Then $\fg=\fn^-\oplus\fh\oplus\fn^+$.   We set $\fb^-=\fn^-\oplus\fh$ and $\fb^+=\fn^+\oplus\fh$. We denote by $U(\fa)$ the universal enveloping algebra of a Lie algebra $\fa$. We denote by $\iota\colon X\to\fhd$ the map that sends $\mu$ to the linear form on $\fh$ given by $h_\alpha\mapsto \lgl\mu,\alpha^\vee\rgl$. Note that $\iota$ is injective if $\ch K=0$, and has kernel $pX$ if $\ch K=p>0$.

\subsection{$X$-graded representations}\label{subsec-Xgradrep}
A convenient framework to deal with arbitrary characteristics is the following.
\begin{definition} \label{def-Xgradrepg} An {\em $X$-graded representation of $\fg$} is an $X$-graded $K$-vector space $M=\bigoplus_{\mu\in X} M_\mu$ that carries the structure of a representation of $\fg$ such that the following are satisfied.
\begin{enumerate}
\item For $\alpha\in\Pi$, the action maps of  $e_\alpha$ and $f_\alpha$ are homogeneous of degree $+\alpha$ and $-\alpha$, resp. Hence $h_\alpha$ acts  homogeneously of degree $0$, and we assume that it acts on $M_\mu$ via the character $\iota(\mu)$. 
\item As a $U(\fb^-)$-module, $M$ is finitely generated. 
\end{enumerate}
We denote by $\CC$ the category that contains all  $X$-graded representations as objects, and has as  morphisms those $\fg$-module homomorphisms that respect the grading.
\end{definition}
\begin{remarks}
\begin{enumerate}
\item If $M$ is an object in $\CC$, then each weight space $M_\mu$ is finite dimensional.
\item 
Suppose that the  characteristic of $K$ is $0$. Then an $X$-graded representation is nothing else but a weight module with integral weights  that is finitely generated over $U(\fb^-)$. 
\item Suppose that the characteristic of $K$ is $p>0$. Then the $X$-grading is a datum additional to the $\fg$-module structure, as $\iota$ is not injective.
\end{enumerate}
\end{remarks}
Let $\lambda\in X$. 
 An object $M$ of $\CC$ is called a {\em highest weight object with highest weight $\lambda$}, if there exists a non-zero vector $v\in M_\lambda$ such that $e_\alpha.v=0$ for all $\alpha\in\Pi$ and $M=U(\fg).v$. In this case, $M_\lambda$ is one-dimensional, $M=U(\fn^-).v$ and $M_\mu\ne \{0\}$ implies $\mu\le\lambda$.
 
 The following is proven using standard arguments.
 
 \begin{lemma} For each $\lambda\in X$ there exists an up to isomorphism unique irreducible object $L(\lambda)$ in $\CC$ with highest weight $\lambda$, and $\{L(\lambda)\}_{\lambda\in X}$ is a full set of representatives of the irreducible objects in $\CC$. 
 \end{lemma}

\subsection{Contravariant duals and contravariant forms}\label{subsec-conform}
From the representation of $\fg$ by generators and relations it follows that there is an antiautomorphism $\tau$ on $\fg$ that interchanges $e_\alpha$ and $f_\alpha$ for all $\alpha\in\Pi$ and restricts to the identity on $\fh$. Let $M$ be an $X$-graded representation. We denote by $dM$ its {\em contravariant dual}. Recall that $dM=\bigoplus_{\mu\in X}M_\mu^\ast\subset M^\ast$ as a vector space, and the action of $\fg$ is given by $(x.f)(m)=f(\tau(x).m)$ for all $f\in dM$, $x\in\fg$ and $m\in M$. Then $dM$ is again an $X$-graded representation if we set  $(dM)_\lambda:=M_\lambda^\ast$.  Note that a homomorphism $M\to dM$ in $\CC$  is the same as a bilinear form $(\cdot,\cdot)\colon M\times M\to K$ that has the following properties:
\begin{itemize}
\item $(x.m,n)=(m,\tau(x).n)$ for all $x\in\fg$, $m,n\in M$,
\item $(m,n)=0$ if $m\in M_\lambda$, $n\in M_\nu$ and $\lambda\ne\mu$.
\end{itemize}
A bilinear form on $M$ with these properties is called a {\em contravariant form} on $M$. 

\begin{lemma}\label{lemma-repconform} Let $\lambda\in X$. 
\begin{enumerate} 
\item  We have $dL(\lambda)\cong L(\lambda)$.
\item There exists a symmetric and non-degenerate contravariant form on $L(\lambda)$.
\end{enumerate}
\end{lemma}

\begin{proof} Note that $dL(\lambda)$ is an irreducible object in $\CC$ as well. As the dimensions of the weight spaces of $L(\lambda)$ and $dL(\lambda)$ coincide, $dL(\lambda)$ has $\lambda$ as its highest weight, hence (1). An isomorphism $L(\lambda)\to dL(\lambda)$ is nothing but 
 a  non-degenerate contravariant form $(\cdot,\cdot)$ on $L(\lambda)$. As $L(\lambda)$ is cyclic as an $U(\fb^-)$-module, the weight decomposition is orthogonal and the highest weight space is one-dimensional, such a form is already determined by its value on a highest weight generator. It follows that it is symmetric. 
 \end{proof}

\subsection{HR-forms and contravariant forms} 
 
Suppose that  $M=\bigoplus_{\mu\in X} M_\mu$, $\{F_\alpha\}_{\alpha\in\Pi}$ satisfy  $(\ast)_K$ .
\begin{theorem} \label{thm-mainK} The following are equivalent:
\begin{enumerate}
\item $M$ admits an HR-form.
\item There exists a structure of an $X$-graded representation of $\fg$ on $M=\bigoplus_{\mu\in X} M_\mu$ such that 
\begin{enumerate}
\item  the action map of $f_\alpha$ is given by $\mbf_\alpha$ for each simple root $\alpha$,
\item  $M$ is semisimple, i.e. isomorphic to a (finite) direct sum of copies of various $L(\lambda)$'s.
\end{enumerate}
\end{enumerate}
\end{theorem}

\begin{proof} We first show that (2) implies (1). Let  $\lambda\in X$. For $\alpha\in\Pi$ we let $F_\alpha\colon L(\lambda)\to L(\lambda)$ be the action map of $f_\alpha\in\fg$. Hence we obtain an $X$-graded space with operators $\{F_\alpha\}$. As $L(\lambda)$ is cyclic as an $U(\fn^-)$-module and as $U(\fn^-)$ is generated by the $f_\alpha$ with $\alpha\in\Pi$, the property $(\ast)_K$ is satisfied. 

By Lemma \ref{lemma-repconform} there exists a non-degenerate and symmetric contravariant form  $(\cdot,\cdot)_{L(\lambda)}$  on $L(\lambda)$. We show that this is an HR-form. As we deal with an $F$-cyclic object of highest weight $\lambda$ and the form is non-degenerate, Lemma \ref{lemma-HRform} implies that it is enough to check properties (1), (3) and (4) in Definition \ref{def-formK}. We have already seen that the form is symmetric. The weight space decomposition is an orthogonal decomposition for any contravariant form. Moreover, for any $\alpha\in\Pi$ the adjoint map $E_\alpha$ to $F_\alpha$  is the action map of  $e_\alpha$. So the commutation relations in Definition \ref{def-formK} follow from the commutation relations between the $e_\alpha$'s and the $f_\beta$'s. So $(\cdot,\cdot)_{L(\lambda)}$ is an HR-form. Now $M$ is isomorphic to a direct sum of copies of various $L(\lambda)$'s, by definition. By the above we obtain an HR-form on each direct summand. By taking the orthogonal direct sum we arrive at an HR-form on $M$.  Hence (2) implies (1).

Now we show that (1) implies (2). By the above we can consider $L(\lambda)$ as an $X$-graded space with $F$-operators. Then $L(\lambda)$ is $F$-cyclic with highest weight $\lambda$, so we obtain an isomorphism 
$\LDelta(\lambda)\cong L(\lambda)$ that commutes with the $F$-operators from Proposition \ref{prop-univprop}. In particular, via transport of structure we obtain the structure of an $X$-graded $\fg$-representation  on $\LDelta(\lambda)$ that makes it into an irreducible representation of  highest weight $\lambda$. 
Now, if $M$ admits an HR-form, then it is isomorphic to a direct sum of copies of various $(\LDelta(\lambda),\{F_{\alpha}\}_{\alpha\in\Pi})$'s by Proposition \ref{prop2-isotyp}. Using the $\fg$-module structure obtained above on each direct summand we see that there indeed exists a $\fg$-module structure on $M$ of the required sort. Hence  (1) implies (2).
\end{proof}

In the next sections we want to generalize the theorem above to a situation in which we replace $K$ by the ring  of $p$-adic numbers. 

\section{Operators in root directions on $X$-graded $\DZ_p$-modules}
Let $R$, $\Pi$ and $X$ be as in Section \ref{sec-Xgradsp}. Let $p$ be any prime number (later we need the restriction $p\ne 2$). 
Suppose now that $M=\bigoplus_{\mu\in X}M_\mu$ is an $X$-graded  $\DZ_p$-module. Suppose that we are given, for any $\alpha\in\Pi$, a $\DZ_p$-linear endomorphism $\mbf_\alpha\colon M\to M$ that is homogeneous of degree $-\alpha$. We will make two general assumptions on these data. The first  is the following.

\begin{itemize}
\item[$(\ast)_{p1}$] $M$ is a free $\DZ_p$-module of finite rank and for any $n\ge 0$ and any $\alpha\in\Pi$, $F_\alpha^n$ maps $M$ into $n!M$.
\end{itemize}
This implies that there is a well defined operator $F_\alpha^{(n)}:=F_\alpha^n/n!$ on $M$. We say that a subset $N$ of $M$ is {\em $F^{(\ast)}$-stable} if $F_\alpha^{(n)}(N)\subset N\cup\{0\}$ for all $n\ge 0$ and $\alpha\in\Pi$. 
 For an arbitrary subset $T$ of $M$ we denote by $\lgl T\rgl_{F^{(\ast)}}$ the smallest $F^{(\ast)}$-stable $\DZ_p$-submodule of $M$. For a closed subset $I$ of $X$ we set $M_I:=\lgl \bigoplus_{\mu\in I} M_\mu\rgl_{F^{(\ast)}}$. Our second assumption is the following.

\begin{enumerate}
\item[$(\ast)_{p2}$] For any closed subset $I$ of $X$, the inclusion $M_I\to M$ splits as an inclusion of $\DZ_p$-modules. 
\end{enumerate}
Note that the above implies that $M_I$  and the cokernel of the inclusion are  free $\DZ_p$-modules of finite rank.

Let $(\cdot,\cdot)\colon M\times M\to\DZ_p$ be a  $\DZ_p$-bilinear form on $M$. 
\begin{definition}\label{def-formp} We say that $(\cdot,\cdot)$ is a {\em $p$-adic HR-form}  if the following holds:
\begin{enumerate}
%
\item It is symmetric and non-degenerate.
\item For any closed subset $I$ of $X$ its restriction to $M_I\times M_I$ is faithful.
\item $M_\mu$ and $M_\nu$ are orthogonal  for $\mu\ne\nu$.
\item For $\alpha\in\Pi$ denote by $E_\alpha\colon M\to M$ the adjoint of $F_\alpha$ with respect to $(\cdot,\cdot)$. Then 
  $[\mbe_\alpha,\mbf_\alpha](v)=\langle\mu,\alpha^\vee\rangle v$ for all $\mu\in X$ and $v\in {M_\mu}$, and
$[\mbe_\alpha,\mbf_\beta]=0$ if $\alpha\ne\beta$.
\end{enumerate}
\end{definition}
Note that the restriction of $(\cdot,\cdot)$  to $M_I\times M_I$ is called faithful if the following holds. If $v\in M_I$ is such that $(v,w)=0$ for all $w\in M_I$, then $v=0$,  i.e. $(\cdot,\cdot)$ induces an {\em injective} homomorphism from $M_I$ into its $\DZ_p$-linear dual.

Denote by $\DQ_p$ the quotient field of $\DZ_p$. For a $\DZ_p$-module $N$ we define $N_{\DQ_p}:=N\otimes_{\DZ_p}\DQ_p$. If $M$ is a $p$-adic $X$-graded space with operators $F_\alpha$, then we can endow $M_{\DQ_p}$ with the induced $X$-grading and the induced operators $F_\alpha$ (that we denote by the same symbol). As $M$ is supposed to be free over $\DZ_p$ we will consider it as a $\DZ_p$-lattice inside $M_{\DQ_p}$. 

Now let us fix a $p$-adic HR-form $(\cdot,\cdot)$  on $M$. We denote by $(\cdot,\cdot)_{\DQ_p}$ the induced $\DQ_p$-bilinear form on $M_{\DQ_p}$.  It is non-degenerate. We denote by  $E_\alpha\colon M_{\DQ_p}\to M_{\DQ_p}$ the homomorphism induced by $E_\alpha\colon M\to M$. It is clearly the adjoint to $F_\alpha\colon M_{\DQ_p}\to M_{\DQ_p}$ with respect to $(\cdot,\cdot)_{\DQ_p}$.
\begin{lemma} \label{lemma-stabE} For any $\alpha\in\Pi$ and $n\ge 0$ we have $E_\alpha^n(M)\subset n! M$.
\end{lemma}
\begin{proof}   Let $v$ be an element in $M_{\DQ_p}$. As $(\cdot,\cdot)$ is non-degenerate on $M$, we have $v\in M$ if and only if $(v,w)_{\DQ_p}\in\DZ_p$ for all $w\in M$. For $v\in M$, $\alpha\in\Pi$ and $n\ge 0$ we have $(E_\alpha^n(v),w)=(v,F_\alpha^n w)\in n!\DZ_p$ for all $w\in M$. We deduce  $E_\alpha^n(v)/n!\in M$ or $E_\alpha^n(v)\in n! M$. 
\end{proof}
We denote by $E_\alpha^{(n)}$ the linear map $E_\alpha^n/n!\colon M\to M$. 

\begin{lemma}\label{lemma-indform} The induced bilinear form $(\cdot,\cdot)_{\DQ_p}$ on $M_{\DQ_p}$ is an HR-form in the sense of Definition \ref{def-formK} for $K=\DQ_p$.
\end{lemma}

\begin{proof} Clearly properties (1), (3) and (4) of Definition \ref{def-formK} are satisfied. Let $I$ be a closed subset of $X$. Recall that we defined $(M_{\DQ_p})_I$ as the smallest $F$-stable $\DQ_p$-subspace in $M_{\DQ_p}$ that contains all $(M_{\DQ_p})_\mu$ with $\mu\in I$. Then $M_I\subset (M_{\DQ_p})_I$ and $M_I$ generates $(M_{\DQ_p})_I$ as a $\DQ_p$-vector space. As the weight spaces of $M_{\DQ_p}$ are mutually orthogonal and finite dimensional, and since the restriction of $(\cdot,\cdot)$ to $M_I$ is faithful, the restriction of $(\cdot,\cdot)_{\DQ_p}$ to $(M_{\DQ_p})_I\times (M_{\DQ_p})_I$ is non-degenerate. Hence property (2) in Definition \ref{def-formK} is also satisfied. 
\end{proof}

We say that a subset $N$ of $M$ is  {\em $E^{(\ast)}$-stable} if $E_\alpha^{(n)}(N)\subset N\cup\{0\}$ for all $\alpha\in\Pi$ and $n\ge 0$. 
\begin{lemma}\label{lemma-substab} Let $I$ be a closed subset of $X$. Then the following holds.
\begin{enumerate}
\item $M_I=M\cap (M_{\DQ_p})_I$.
\item $M_I$ is $F^{(\ast)}$- and $E^{(\ast)}$-stable.
\end{enumerate}
\end{lemma}
\begin{proof} Note that $M_I\subset (M_{\DQ_p})_I$ and $M_I$ generates $(M_{\DQ_p})_I$ as a $\DQ_p$-vector space. It follows that the cokernel of the inclusion $M_I\subset M\cap (M_{\DQ_p})_I$ is a torsion $\DZ_p$-module. But since we assume that the inclusion $M_I\subset M$ splits, this torsion module must vanish. Hence (1). Now (2) is a consequence of (1), as both $M$ and $(M_{\DQ_p})_I$ are $F^{(\ast)}$- and  $E^{(\ast)}$-stable (cf. Lemma \ref{lemma-orthdec}). 
\end{proof}

\section{The connection to representation theory over $\DZ_p$}

The algebraic structure the we utilize in this section is the $\DZ_p$-algebra of distributions (also called  the $\DZ_p$-hyperalgebra) associated with $R$. It is defined as follows. Let $\fg_\DQ$ be the semisimple Lie algebra over $\DQ$ associated with $R$ (see the introduction to Section \ref{sec-repthK}) and $\fh_{\DQ}$ its Cartan subalgebra. For any $x\in\fg_\DQ$ and $n\in\DN$ define
$$
x^{(n)}:=\frac{x^n}{n!}, \quad {x\choose n}:=\frac{x(x-1)\cdots(x-n+1)}{n!}\in U(\fg_{\DQ}). 
$$
Now fix a Chevalley system $\{X_\gamma\}_{\gamma\in R}$ of $(\fg_{\DQ},\fh_{\DQ})$ (see \cite[Chapter VIII, 2.4]{B}). Recall that this means that $X_\gamma$ is an element in $\fg_{\DQ}$ that has weight $\gamma$ with respect to the adjoint action of $\fh_{\DQ}$,  and  $[X_\gamma,X_{-\gamma}]=\gamma^\vee\in\fh_{\DQ}$ is the coroot of $\gamma$ for all $\gamma\in R$. Moreover, it means that the $\DQ$-linear automorphism $\tau\colon\fg_{\DQ}\to\fg_{\DQ}$ that maps $X_\gamma$ to $X_{-\gamma}$ and is the identity on $\fh$ is a Lie algebra-antiautomorphism (note that our sign convention differs slightly from \cite[Chapter VIII, 2.4]{B} and hence our $\tau$ is an {\em anti}automorphism).

\subsection{The algebra of distributions}
Denote by $\Dist_\DZ\subset U(\fg_\DQ)$ the unital  $\DZ$-subalgebra that is generated by $X_{\gamma}^{(n)}$ with $\gamma\in R$ and $n\ge 0$, and ${\alpha^\vee\choose n}$ with $\alpha\in\Pi$ and $n\ge 0$, and set $\Dist_{\DZ_p}:=\Dist_\DZ\otimes_\DZ \DZ_p$. 

\begin{definition} We denote by $\cD$ the full subcategory of the category of all $\Dist_{\DZ_p}$-modules that contains all objects $M$ that satisfy the following.
\begin{enumerate}
\item As a $\DZ_p$-module, $M$ is free of finite rank. 
\item For $\alpha\in\Pi$, the element $h_\alpha\in\fg_{\DQ}$ acts diagonalizably on $M$ with integral eigenvalues.
\end{enumerate}
\end{definition}
If $M$ is an object in $\cD$, then we denote by $M=\bigoplus_{\mu\in X}M_\mu$ its weight decomposition with respect to the action of the $h_{\alpha}$, $\alpha\in\Pi$. So $h_{\alpha}$ acts on $M_\mu$ by multiplication with $\lgl\mu,\alpha^\vee\rgl$. 

\begin{remark} Let $G_{\DZ_p}$ be the semisimple, connected and simply connected split algebraic group scheme over $\DZ_p$ with root system $R$. By \cite[Section II, 1.20]{J}, an object in $\cD$  is the same as a rational representation of $G_{\DZ_p}$ that is free of finite rank as a $\DZ_p$-module. A morphism in $\cD$ is then the same as a homomorphism of  $G_{\DZ_p}$-modules. 
\end{remark}

 The following   is proven in \cite[Chapter VIII, Section 12]{B}. For $\alpha\in\Pi$ we write $e_\alpha:=X_\alpha$ and $f_\alpha:=X_{-\alpha}$.

\begin{proposition}\label{prop-Distgens} $\Dist_{\DZ}$ is generated as a $\DZ$-algebra by $e_\alpha^{(n)}$ and $f_{\alpha}^{(n)}$ for $\alpha\in\Pi$ and $n\ge 0$, and $\Dist_{\DZ}^-:=\Dist_\DZ\cap U(\fn^-)$ is generated by the elements $f_{\alpha}^{(n)}$ for $\alpha\in\Pi$ and $n\ge0$. 
\end{proposition}

\subsection{Contravariant forms} 
Recall the antiautomorphism $\tau\colon \fg_{\DQ}\to \fg_{\DQ}$ from above. It induces an antiautomorphism on $\Dist_{\DZ_p}$ that we denote by the same symbol. 
For an object $M$ in $\cD$ we denote by $dM$ its {\em contravariant dual}. As a $\DZ_p$-module, $(dM)_\lambda=(M_\lambda)^\ast=\Hom_{\DZ_p}(M_\lambda,\DZ_p)$ and $x\in \Dist_{\DZ_p}$ acts on $\phi\in M^\ast$ via $x.\phi=\phi\circ \tau(x)$ (cf. \cite[Section II.2.12]{J}, where $dM$ is denote by ${}^\tau M$).
A homomorphism $\phi\colon M\to dM$ is the same as a bilinear form $(\cdot,\cdot)\colon M\times M\to \DZ_p$ that has the property that $(x.v,w)=(v.\tau(x).w)$ for all $v,w\in M$ and $x\in \Dist_{\DZ_p}$. A bilinear form on $M$ with this property is called a {\em contravariant form} on $M$. 

\subsection{Weyl modules}
We denote by $X^+\subset X$ the subset of dominant weights. 
For a dominant weight $\lambda$ we denote by $L_{\DQ_p}(\lambda)$ the simple highest weight module for $\fg_{{\DQ_p}}$ with highest weight $\lambda$. It is finite dimensional.  
 Let $v\in L_{\DQ_p}(\lambda)_\lambda$ be a non-zero element.  Set $\WDelta_{\DZ_p}(\lambda):=\Dist_{\DZ_p}.v\subset L_{\DQ_p}(\lambda)$. This  is called the {\em Weyl module for $\Dist_{\DZ_p}$ with highest weight $\lambda$}. It is a weight module with integral weights, and  free as $\DZ_p$-module of finite rank. So it is an object in $\cD$ (cf. \cite[Part II, Section 8.3]{J}). Note that $\Delta_{\DZ_p}(\lambda)$ is cyclic as a $\Dist^-_{\DZ_p}$-module (this follows from the version of the PBW-theorem for the integral distribution algebra, cf. \cite[Part II, Section 1.12]{J}).
 
\begin{lemma} \label{lemma-Weylisotyp}  Let $M$ be an object in $\cD$. Suppose that $\lambda$ is maximal among the weights of $M$ and that $M$ is as a $\Dist_{\DZ_p}$-module generated by $M_\lambda$. Then $M$ is isomorphic to a direct sum of Weyl modules with highest weight $\lambda$.
\end{lemma}

\begin{proof} We consider $M$ as a $\DZ_p$-lattice in $M_{\DQ_p}$. Note that $M_{\DQ_p}$ is a $\Dist_{\DZ_p}\otimes_{\DZ_p}\DQ_p=U(\fg_{\DQ_p})$-module. Then $M_{\DQ_p}$ is a finite dimensional weight module generated by its $\lambda$-weight space. Hence $M_{\DQ_p}$ is isomorphic to a direct sum of copies of $L_{\DQ_p}(\lambda)$. Now fix  a basis $\{m_1,\dots,m_r\}$ of the free $\DZ_p$-module $M_\lambda$ and set $M_i:=\Dist_{\DZ_p}.m_i$. As $U(\fg_{\DQ_p}).m_i$ must be isomorphic to  $L_{\DQ_p}(\lambda)$, we obtain that  $M_i$ is isomorphic to $\Delta_{\DZ_p}(\lambda)$. Moreover, as  $\{m_1,\dots,m_r\}$ is a basis of the $\DQ_p$-vector space $(M_{\DQ_p})_\lambda$ we deduce  $M_{\DQ_p}=\bigoplus_{i=1}^r U(\fg_{\DQ_p}).m_i$. So the sum of the $M_i$ inside $M$ must be direct. As $M$ is generated by $M_\lambda$, we obtain $M=\bigoplus_{i=1}^rM_i\cong\bigoplus_{i=1}^r\Delta_{\DZ_p}(\lambda)$. 
\end{proof}
\subsection{Weyl filtrations}
We show in this section that our property $(\ast)_{p2}$ is connected to the following representation theoretic notion.

\begin{definition} Let $M$ be a $\Dist_{\DZ_p}$-module.  A {\em Weyl filtration} of $M$ is a finite filtration by $\Dist_{\DZ_p}$-submodules such that the subquotients are isomorphic to Weyl modules. 
\end{definition}
Note that since Weyl modules are free of finite rank as $\DZ_p$-modules, so is each module with a Weyl filtration.
For a closed subset $I$ of $X$ define $M[I]\subset M$ as the $\Dist_{\DZ_p}$-submodule that is generated by $\bigoplus_{\gamma\in I}M_\gamma$.

\begin{proposition}\label{prop-tilt1}  Let $M$ be an object in $\cD$.  Then the following are equivalent.
\begin{enumerate}
\item For each closed subset $I$, the inclusion $M[I]\subset M$ splits as an inclusion of $\DZ_p$-modules. 
\item $M$ admits a Weyl filtration. 
\end{enumerate}
\end{proposition}
\begin{proof} We assume (1). Let $\mu_1$, \dots, $\mu_r$ be an enumeration of all weights of $M$  with the property that $\mu_i<\mu_j$ implies $j<i$. For $j=0,\dots,r$ set $I_j=\bigcup_{i\le j}\{\ge\mu_i\}$. Then  $\emptyset=I_0\subset I_1\subset\dots\subset I_r$ and all $I_j$ are closed.
From (1)  it follows that each inclusion $M[{I_{j-1}}]\subset M[{I_{j}}]$ splits and hence the quotient  $M[{I_{j}}]/M[I_{j-1}]$ is free over $\DZ_p$. By construction, it is generated by its $\mu_j$-weight space, and Lemma \ref{lemma-Weylisotyp} implies that it is isomorphic to a direct sum of copies of $\Delta_{\DZ_p}(\mu_j)$. So we can refine the filtration $0=M[{I_0}]\subset M[{I_1}]\subset\dots \subset M[{I_r}]=M$ and obtain a Weyl filtration. 

Now we assume (2). Let $I\subset X$ be closed. As $\Ext^1(\Delta_{\DZ_p}(\lambda),\Delta_{\DZ_p}(\mu))=0$ unless $\mu>\lambda$, we can reorder a Weyl filtration in such a way that we find a filtration $0=M_0\subset \dots\subset M_r=M$ with $M_i/M_{i-1}\cong\Delta_{\DZ_p}(\lambda_i)$ such that $\lambda_1,\dots,\lambda_l\in I$ and $\lambda_{l+1},\dots,\lambda_r\not\in I$ for some $l\in\{0,...,r\}$. Hence $M[I]=M_l$ and $M/M[I]=M/M_l$ admit Weyl filtrations as well. In particular, $M/M[I]$ is free as a $\DZ_p$-module, and hence the inclusion $M[I]\subset M$ splits. 
\end{proof}

\subsection{Tilting modules}
Let $T$ be an object in $\cD$.

\begin{definition} $T$ is called a tilting module if it admits a Weyl filtration and so does its contravariant dual $dT$. 
\end{definition}

For any dominant weight $\lambda$ there is an up to isomorphism unique tilting module $T_{\DZ_p}(\lambda)$ which has the following properties: it is indecomposable, its $\lambda$-weight space $T_{\DZ_p}(\lambda)_\lambda$ is free of rank $1$, and $T_{\DZ_p}(\lambda)_\mu\ne 0$ implies $\mu\le \lambda$ (cf. \cite[Section E.19]{J}).

\begin{remark} Let $k$ be an algebraically closed field of characteristic $p$. For any $\lambda\in X^+$ there exists an analogous Weyl module $\WDelta_k(\lambda)$ and a  tilting module $T_k(\lambda)$ for the connected semisimple, simply connected algebraic group $G_k$ with root system $R$. The number of occurrences of $\WDelta_k(\mu)$ in a given Weyl filtration of $T_k(\lambda)$ is independent of the filtration and  is of central importance in representation theory.   By \cite[Section E.23]{J}, one can obtain $T_k(\lambda)$ from $T_{\DZ_p}(\lambda)$ by extension of scalars. The same is true for Weyl modules, and the Weyl multiplicities in $T_k(\lambda)$ and $T_{\DZ_p}(\lambda)$ coincide.
\end{remark}

\begin{lemma}\label{lemma-symformtilt} Let $\lambda\in X^+$. 
\begin{enumerate}
\item  We have $dT_{\DZ_p}(\lambda)\cong T_{\DZ_p}(\lambda)$.
\item If $p\ne 2$, then there exists a non-degenerate and  symmetric contravariant form $(\cdot,\cdot)$ on $T_{\DZ_p}(\lambda)$.
\end{enumerate}
\end{lemma}

\begin{proof} (1) As the contravariant duality is additive and satisfies $d^2\cong\id_{\cD}$, it follows  that $dT_{\DZ_p}(\lambda)\cong T_{\DZ_p}(\mu)$ for some $\mu\in X^+$. But the dimensions of the weight spaces  of $M$ and $dM$ coincide for all $M$ in $\cD$. A comparison of highest weights hence yields $\mu=\lambda$. 

(2) By (1), there exists a non-degenerate contravariant form $(\cdot,\cdot)^\prime$ on $T_{\DZ_p}(\lambda)$. Denote by $(\cdot,\cdot)$ the sum of $(\cdot,\cdot)^\prime$ and its transpose. This is now a symmetric contravariant form. As $T_{\DZ_p}(\lambda)_\lambda$ is free of rank $1$, since the restriction of $(\cdot,\cdot)^\prime$ to $T_{\DZ_p}(\lambda)_\lambda$ is non-degenerate, and since $p\ne 2$, the restriction of $(\cdot,\cdot)$ to $T_{\DZ_p}(\lambda)_\lambda$ is non-degenerate as well. By \cite[Section E.20]{J}, $(\cdot,\cdot)$ is hence non-degenerate.
\end{proof}

\begin{proposition} \label{prop-ontilt} Let $M$ be an object in $\cD$. The following are equivalent.
\begin{enumerate}
\item  For any closed subset $I$, the inclusion $M[I]\subset M$ splits as an inclusion of $\DZ_p$-modules, and $M$ is isomorphic to its contravariant dual. 
\item $M$ is a tilting module. 
\end{enumerate}
\end{proposition}
\begin{proof} Assume that property (1) is satisfied.  Proposition \ref{prop-tilt1} implies that $M$ admits a Weyl filtration. As it is self-dual, also its dual admits a Weyl filtration.  

Now assume property (2). As $M$ admits a Weyl filtration, Proposition \ref{prop-tilt1} implies that the inclusion $M[I]\subset M$ splits for any closed subset $I$. As each tilting module is self-dual with respect to the contravariant duality, property (1) is satisfied. 
\end{proof}

\subsection{Tilting modules as graded spaces with operators}
In this section we assume $p\ne2$. The only reason for this assumption is that we need the existence of a {\em symmetric} contravariant form on tilting modules, see Lemma \ref{lemma-symformtilt}. Let $T=\bigoplus_{\mu\in X}T_\mu$ be a tilting module in $\cD$. For any $\alpha\in\Pi$ denote by $F_\alpha\colon T\to T$ the action map of the element $f_\alpha\in\Dist_{\DZ_p}$. Fix a symmetric non-degenerate contravariant form $(\cdot,\cdot)$ on $T$ (cf. Lemma \ref{lemma-symformtilt}).

\begin{proposition} The data $T=\bigoplus_{\mu\in X}T_\mu$ and $\{F_\alpha\colon T\to T\}_{\alpha\in\Pi}$ satisfy $(\ast)_{p1}$ and $(\ast)_{p2}$, and $(\cdot,\cdot)$ is a $p$-adic HR-form on $T$.
\end{proposition}
\begin{proof} Clearly, $T$ is free over $\DZ_p$ of finite rank. 
For any $\alpha\in\Pi$ and $n\ge 0$ we have $F_{\alpha}^n=n!F_{\alpha}^{(n)}$ which implies $F_\alpha^n(T)\subset n!T$. Hence $(\ast)_{p1}$ is satisfied. As each Weyl module is cyclic as a $\Dist^-_{\DZ_p}$-module, generated by its highest weight, we can use Proposition \ref{prop-Distgens} to deduce that $T[I]=T_I$ for any closed subset $I$ of $X$. Then  Proposition \ref{prop-ontilt} implies that $T_I\subset T$ splits as an inclusion of  $\DZ_p$-modules. 
 Hence $(\ast)_{p2}$ is satisfied as well. 

It remains to show that the contravariant form $(\cdot,\cdot)$ is a $p$-adic  HR-form. It is symmetric and non-degenerate, and the weight spaces are orthogonal. The commutation relations between the $F_\alpha$ and their adjoints are implied by the commutation relations between the $f_\alpha$'s and the $e_\beta$'s. Now let $I\subset X$ be a closed subset. Note that $M_{\DQ_p}$ is a semisimple representation of $U(\fg_Q)$, and $(M_I)_{\DQ_p}$ is the subrepresentation that contains all isotypic components with highest weights in  $I$. The non-degenerate contravariant form on $M$ induces a  non-degenerate form on $M_{\DQ_p}$. As the decomposition into isotypic components is an orthogonal decomposition with respect to an arbitrary contravariant form, we deduce that the restriction to $(M_I)_{\DQ_p}$ is non-degenerate. This implies that the restriction to $M_I\subset(M_{\DQ_p})_I$ is faithful. Hence conditions (1)--(4) of Definition \ref{def-formp} are satisfied.
\end{proof}

In the next section we prove the converse of the above, i.e. the existence of an HR-form on a graded space $M$ with simple root operators ensures that there is the structure of a tilting module on $M$.

\subsection{Tilting modules from graded spaces with operators}

In the following we do not have to assume that $p$ is odd.

\begin{theorem} Suppose that $M=\bigoplus_{\mu\in X}M_\mu$ and $\{F_\alpha\colon M\to M\}_{\alpha\in\Pi}$ satisfy the conditions $(\ast)_{p1}$ and $(\ast)_{p2}$. Suppose that there exists a $p$-adic HR-form on $M$. Then there exists a structure of  a $\Dist_{\DZ_p}$-module on $M$ such that $F_\alpha$ is the action map of $f_\alpha$ for all $\alpha\in\Pi$. Moreover, with this structure,  $M$ is an object in $\cD$ and a tilting module. \end{theorem}

\begin{proof}
We fix a $p$-adic HR-form $(\cdot,\cdot)$ on $M$.
We  denote by $(\cdot,\cdot)_{{\DQ_p}}\colon M_{\DQ_p}\times M_{\DQ_p}\to {\DQ_p}$ the bilinear from induced by the HR-form $(\cdot,\cdot)$. By Lemma \ref{lemma-indform}, this is an HR-form in the sense of Definition \ref{def-formK} on $M_{\DQ_p}$. Theorem \ref{thm-mainK}  hence yields a structure of an $X$-graded $\fg_{\DQ_p}$-representation on $M_{\DQ_p}$ such that the action maps of $f_\alpha$ and $e_\alpha$ are  $F_\alpha$ and $E_\alpha$.   Now the $\DZ_p$-lattice $M\subset M_{\DQ_p}$ is stable under the operators $f_{\alpha}^{(n)}$ and $e_{\alpha}^{(n)}$, so  Proposition \ref{prop-Distgens} implies that it is  a $\Dist_{\DZ_p}$-submodule. 
 
  It remains to prove that $M$, endowed with this structure, is a tilting module.  First note that the HR-form becomes a non-degenerate contravariant form on $M$. Hence $M\cong dM$. According to Proposition \ref{prop-ontilt} it hence suffices to show that for any closed subset $I$ of $X$ the inclusion $M[I]\subset M$ splits as an inclusion of $\DZ_p$-modules. So let $I$ be a closed subset of $X$. Then $M_I\subset M$ splits as an inclusion of $\DZ_p$-modules, by Assumption $(\ast)_{p2}$. But, by Lemma \ref{lemma-substab}, $M_I$ is $F^{(\ast)}$- and $E^{(\ast)}$-stable. So it is a $\Dist_{\DZ_p}$-submodule and hence must coincide with $M[I]$.\end{proof}
 
\end{document}